\documentclass[12pt]{article}
\usepackage{pifont}

\usepackage{amsfonts}
\usepackage{latexsym}
\usepackage{amsmath}
\usepackage{amssymb}
\usepackage{color}

\newtheorem{theorem}{Theorem}[section]

\newtheorem{proposition}{Proposition}[section]

\newtheorem{example}{Example}[section]

\newenvironment{proof}[1][Proof]{\noindent \textbf{#1.} }{\ \ \  $\Box$}

\newtheorem{remark}{Remark}[section]

\title{A general comparison theorem for $1$-dimensional anticipated BSDEs}

\date{}

 \author{ Xiaoming Xu\thanks{E-mail: xmxu@mail.sdu.edu.cn}
 \\ \small{School of Mathematical Sciences, Nanjing Normal University, Nanjing, 210046, China}\\
 \small{School of Mathematics, Shandong University, Jinan, 250100, China}
 }

\begin{document}

\maketitle

\begin{abstract}
Anticipated backward stochastic differential equation (ABSDE)
studied the first time in 2007 is a new type of stochastic
differential equations. In this paper, we establish a general
comparison theorem for $1$-dimensional ABSDEs with the generators
depending on the anticipated term of $Z$.
\\
\par $\textit{Keywords:}$ Anticipated backward stochastic
differential equation, Backward stochastic differential equation,
Comparison theorem
\end{abstract}



\section{Introduction}\label{sec:intro}

Backward Stochastic Differential Equation (BSDE) of the following
general form was considered the first time by Pardoux-Peng
\cite{PP1} in 1990:
\begin{center}
$Y_t=\xi+\int_t^T g(s, Y_s, Z_s)ds-\int_t^T Z_sdB_s.$
\end{center}
Since then, the theory of BSDEs has been studied with great
interest. One of the achievements of this theory is the comparison
theorem, which is due to Peng \cite{P} and then generalized by
Pardoux-Peng \cite{PP2}, El Karoui-Peng-Quenez \cite{KPQ} and
Hu-Peng \cite{HP}. It allows to compare the solutions of two BSDEs
whenever we can compare the terminal conditions and the generators.

Recently, a new type of BSDE, called anticipated BSDE (ABSDE in
short), was introduced by Peng-Yang \cite{PY} (see also Yang
\cite{Y}). The ABSDE is of the following form:
$$
\left\{
\begin{tabular}{rll}
$-dY_t=$ & $f(t, Y_t, Z_t, Y_{t+\delta(t)},
Z_{t+\zeta(t)})dt-Z_tdB_t, $ & $
t\in[0, T];$\\
$Y_t=$ & $\xi_t, $ & $t\in[T, T+K];$\\
$Z_t=$ & $\eta_t, $ & $t\in[T, T+K],$
\end{tabular}\right.\eqno(1.1)$$
where $\delta(\cdot): [0, T]\rightarrow \mathbb{R}^+ \setminus
\{0\}$ and $\zeta(\cdot): [0, T]\rightarrow \mathbb{R}^+\setminus
\{0\}$ are continuous functions satisfying
\par $\mathbf{(a1)}$ there exists a constant $K \geq 0$ such that for each $s\in[0,
T],$ $s+\delta(s) \leq T+K$, $s+\zeta(s) \leq T+K;$
\par $\mathbf{(a2)}$ there exists a constant $M \geq 0$ such that
for each $t\in[0, T]$ and each nonnegative integrable function
$g(\cdot),$ $\int_t^T g(s+\delta(s))ds\leq M\int_t^{T+K} g(s)ds$,
$\int_t^T g(s+\zeta(s))ds\leq M\int_t^{T+K} g(s)ds.$

Peng and Yang proved in \cite{PY} that (1.1) has a unique adapted
solution under proper assumptions, furthermore, they established a
comparison theorem, which requires that the generators of the ABSDEs
cannot depend on the anticipated term of $Z$ and one of them must be
increasing in the anticipated term of $Y$.

The aim of this paper is to give a more general comparison theorem
in which the generators of the ABSDEs break through the above
restrictions. The main approach we adopt is to consider an ABSDE as
a series of BSDEs and then apply the well-known comparison theorem
for $1$-dimensional BSDEs (see \cite{KPQ}).

The paper is organized as follows: in Section 2, we list some
notations and some existing results. In Section 3, we mainly study
the comparison theorem for ABSDEs.

\section{Preliminaries}\label{sec:pre}

Let $\{B_t; t\geq 0\}$ be a $d$-dimensional standard Brownian motion
on a probability space $(\Omega, \mathcal{F}, P)$ and
$\{\mathcal{F}_t; t\geq 0\}$ be its natural filtration. Denote by
$|\cdot|$ the norm in $\mathbb{R}^m.$ Given $T
>0,$ we make the following notations:

$L^2(\mathcal{F}_T; \mathbb{R}^m)$ = $\{\xi\in \mathbb{R}^m$ $|$
$\xi$ is an $\mathcal{F}_T$-measurable random variable such that
$E|\xi|^2< \infty\};$

$L_{\mathcal{F}}^2(0, T; \mathbb{R}^m)$ = $\{ \varphi: \Omega\times
[0, T]\rightarrow \mathbb{R}^m$ $|$ $\varphi$ is progressively
measurable; $E\int_0^T |\varphi_t|^2dt< \infty\};$

$S_{\mathcal{F}}^2(0, T; \mathbb{R}^m)$ = $\{\psi: \Omega\times [0,
T]\rightarrow \mathbb{R}^m$ $|$ $\psi$ is continuous and
progressively measurable; $E[\sup_{0 \leq t \leq T} |\psi_t|^2]<
\infty\}.$

Now consider the ABSDE (1.1). First for the generator $f(\omega, s,
y, z, \theta, \phi): \Omega \times [0, T]\times \mathbb{R}^m\times
\mathbb{R}^{m\times d}\times S_\mathcal{F}^2(s, T+K;
\mathbb{R}^m)\times L_\mathcal{F}^2(s, T+K; \mathbb{R}^{m\times
d})\rightarrow L^2 (\mathcal{F}_s; \mathbb{R}^m),$ we use two
hypotheses:

${\bf{(H1)}}$ there exists a constant $L > 0$ such that for each
$s\in [0, T],$ $y, y^\prime\in \mathbb{R}^m,$ $z, z^\prime \in
\mathbb{R}^{m\times d},$ $\theta, \theta^\prime \in
L_\mathcal{F}^2(s, T+K; \mathbb{R}^m),$ $\phi, \phi^\prime \in
L_\mathcal{F}^2(s, T+K; \mathbb{R}^{m\times d}),$ $r, \bar{r}\in [s,
T+K],$ the following holds: $$|f(s, y, z, \theta_r,
\phi_{\bar{r}})-f(s, y^\prime, z^\prime, \theta_r^\prime,
\phi_{\bar{r}}^\prime)|\leq L(|y-
y^\prime|+|z-z^\prime|+E^{\mathcal{F}_s}[|\theta_r-\theta_r^\prime|+|\phi_{\bar{r}}-\phi_{\bar{r}}^\prime|]);
$$

${\bf{(H2)}}$ $E[\int_0^T |f(s, 0, 0, 0, 0)|^2ds]< \infty.$

Let us review the existence and uniqueness theorem for ABSDEs from
\cite{PY}:

\begin{theorem}\label{Theorem 2.1}
Assume that $f$ satisfies (H1) and (H2), $\delta,$ $\zeta$ satisfy
(a1) and (a2), then for arbitrary given terminal conditions $(\xi,
\eta) \in S_\mathcal{F}^2(T, T+K; \mathbb{R}^m)\times
L_\mathcal{F}^2(T, T+K; \mathbb{R}^{m\times d}),$ the ABSDE (1.1)
has a unique solution, i.e., there exists a unique pair of
$\mathcal{F}_t$-adapted processes $(Y, Z)\in S_\mathcal{F}^2(0, T+K;
\mathbb{R}^m)\times L_\mathcal{F}^2(0, T+K; \mathbb{R}^{m\times d})$
satisfying (1.1).
\end{theorem}

Next we will recall the comparison theorem from \cite{PY}. Let
$(Y^{(j)}, Z^{(j)})$ $(j=1, 2)$ be solutions of the following
$1$-dimensional ABSDEs respectively:
$$
\left\{
\begin{tabular}{rll}
$-dY_t^{(j)}=$ & $f_j(t, Y_t^{(j)}, Z_t^{(j)},
Y_{t+\delta(t)}^{(j)})dt-Z_t^{(j)}dB_t, $ & $
t\in[0, T];$\\
$Y_t^{(j)}=$ & $\xi_t^{(j)}, $ & $t\in[T, T+K].$
\end{tabular}\right.\eqno(2.1)
$$

\begin{theorem}\label{Theorem 2.2}
Assume that $f_1,$ $f_2$ satisfy (H1) and (H2), $\xi^{(1)},
\xi^{(2)} \in S_\mathcal{F}^2(T, T+K; \mathbb{R}),$ $\delta$
satisfies (a1), (a2), and for each $t\in [0, T],$ $y\in \mathbb{R},$
$z\in \mathbb{R}^d,$ $f_2(t, y, z, \cdot)$ is increasing, i.e.,
$f_2(t, y, z, \theta_r)\geq f_2(t, y, z, \theta_r^\prime)$, if
$\theta_r\geq \theta_r^\prime,$ $\theta, \theta^\prime\in
L_\mathcal{F}^2(t, T+K; \mathbb{R}), r\in [t, T+K].$ If
$\xi_s^{(1)}\geq \xi_s^{(2)}, s\in [T, T+K]$ and $f_1(t, y, z,
\theta_r)\geq f_2(t, y, z, \theta_r), t\in[0, T], y\in \mathbb{R},
z\in \mathbb{R}^d, \theta\in L_\mathcal{F}^2(t, T+K; \mathbb{R}),
r\in [t, T+K],$ then $Y_t^{(1)}\geq Y_t^{(2)},\ a.e.,a.s..$
\end{theorem}

\section{Comparison Theorem for Anticipated BSDEs}\label{secsec:bsde}

Consider the following $1$-dimensional ABSDEs:
$$
\left\{
\begin{tabular}{rll}
$-dY_t^{(j)}=$ & $f_j(t, Y_t^{(j)}, Z_t^{(j)},
Y_{t+\delta(t)}^{(j)}, Z_{t+\zeta(t)}^{(j)})dt-Z_t^{(j)}dB_t, $ & $
t\in[0, T];$\\
$Y_t^{(j)}=$ & $\xi_t^{(j)}, $ & $t\in [T, T+K];$\\
$Z_t^{(j)}=$ & $\eta_t^{(j)}, $ & $t\in [T, T+K],$
\end{tabular}\right.\eqno(3.1)
$$
where $j=1, 2,$ $f_j$ satisfies (H1), (H2), $(\xi^{(j)},
\eta^{(j)})\in S_\mathcal{F}^2(T, T+K; \mathbb{R})\times
L_\mathcal{F}^2(T, T+K; \mathbb{R}^{d}),$ $\delta, \zeta$ satisfy
(a1) and (a2). By Theorem \ref{Theorem 2.1}, either of the above
ABSDEs has a unique adapted solution.

\begin{proposition}\label{Proposition 3.1}
Putting $t_0=T,$ we define by iteration
\begin{center}
$ t_i:=\min\{t \in [0, T]: \min\{s+ \delta(s),\ s+\zeta(s)\}\geq
t_{i-1},\ for\ all\ s\in [t, T]\},\ \ i\geq 1. $
\end{center}
Set $N:=\max\{i: t_{i-1} > 0\}$. Then $N$ is finite, $t_N=0$ and
\begin{center}
$ [0, T]=[0, t_{N-1}] \cup [t_{N-1}, t_{n-2}] \cup \cdots \cup [t_2,
t_1] \cup [t_1, T]. $
\end{center}
\end{proposition}

\begin{proof}
Let us first prove that $N$ is finite. For this purpose, we apply
the method of reduction to absurdity. Suppose $N$ is infinite. From
the definition of $\{t_i\}_{i=1}^{+\infty}$, we know
$$
\min\{t_i+ \delta(t_i),\ t_i+ \zeta(t_i)\}=t_{i-1},\ i=1, 2,
\cdots.\eqno(3.2)
$$
Since $\delta(\cdot)$ and $\zeta(\cdot)$ are continuous and
positive, thus obviously we have $t_i<t_{i-1}$ $(i=1, 2, \cdots).$
Therefore $\{t_i\}_{i=1}^{+\infty}$ converges as a strictly monotone
and bounded series. Denote its limit by $\bar{t}.$ Letting
$i\rightarrow +\infty$ on both sides of (3.2), we get
\begin{center}
$\min\{\bar{t}+\delta(\bar{t}),\ \bar{t}+\zeta(\bar{t})\}=\bar{t}.$
\end{center}
Hence $\delta(\bar{t})=0$ or $\zeta(\bar{t})=0$, which is just a
contradiction since both $\delta$ and $\zeta$ are positive.

Next we will show that $t_N=0.$ In fact, the following holds
obviously: $$\min\{t_N+\delta(t_N),\ t_N+\zeta(t_N)\}> t_N,$$ which
implies $t_N = 0,$ or else  we can find a $\tilde{t}\in [0, t_N)$
due to the continuity of $\delta(\cdot)$ and $\zeta(\cdot)$ such
that
$$\min\{s+\delta(s),\ s+\zeta(s)\}\geq t_N,\ for\ all\ s\in [\tilde{t},
T],$$ from which we know that $\tilde{t}$ is an element of the
series as well.
\end{proof}

\begin{proposition}\label{Proposition 3.2}
Suppose $(Y^{(j)}, Z^{(j)})$
$(j=1, 2)$ are the solutions of ABSDEs (3.1) respectively. Then for
fixed $i\in {\{1, 2, \cdots, N\}},$ over time interval $[t_i,
t_{i-1}]$, ABSDEs (3.1) are equivalent to the following ABSDEs:
$$
\left\{
\begin{tabular}{rll}
$-d\bar{Y}_t^{(j)}=$ & $f_j(t, \bar{Y}_t^{(j)}, \bar{Z}_t^{(j)},
\bar{Y}_{t+\delta(t)}^{(j)},
\bar{Z}_{t+\zeta(t)}^{(j)})dt-\bar{Z}_t^{(j)}dB_t, $ & $
t\in[t_i, t_{i-1}];$\\
$\bar{Y}_t^{(j)}=$ & $Y_t^{(j)}, $ & $t\in [t_{i-1}, T+K];$\\
$\bar{Z}_t^{(j)}=$ & $Z_t^{(j)}, $ & $t\in [t_{i-1}, T+K],$
\end{tabular}\right.\eqno(3.3)
$$
which are also equivalent to the following BSDEs with terminal
conditions $Y_{t_{i-1}}^{(j)}$ respectively:
$$
\tilde{Y}_t^{(j)}=Y_{t_{i-1}}^{(j)}+\int_t^{t_{i-1}} f_j(s,
\tilde{Y}_s^{(j)}, \tilde{Z}_s^{(j)}, Y_{s+\delta(s)}^{(j)},
Z_{s+\zeta(s)}^{(j)})ds-\int_t^{t_{i-1}}
\tilde{Z}_s^{(j)}dB_s.\eqno(3.4)$$ That is to say,
$$Y_t^{(j)}=\bar{Y}_t^{(j)}=\tilde{Y}_t^{(j)},\
Z_t^{(j)}=\bar{Z}_t^{(j)}=\tilde{Z}_t^{(j)}=\frac{d\langle
\tilde{Y}^{(j)}, B\rangle_t}{d t},\ t\in [t_i, t_{i-1}],\ j=1,2,$$
where $\langle \tilde{Y}^{(j)}, B\rangle$ is the variation process
generated by $\tilde{Y}^{(j)}$ and the Brownian motion $B$.
\end{proposition}

\begin{proof}
We only need to prove the equivalence between
ABSDE (3.3) and BSDE (3.4). It is obvious that for each $s\in [t_i,
t_{i-1}]$, $s+\delta(s)\geq t_{i-1}$, $s+\zeta(s)\geq t_{i-1}$, thus
$(\bar{Y}_{t+\delta(t)}^{(j)},
\bar{Z}_{t+\zeta(t)}^{(j)})=(Y_{t+\delta(t)}^{(j)},
Z_{t+\zeta(t)}^{(j)})$ in the generator of ABSDE (3.3). Clearly
$f_j(\cdot, \cdot, \cdot, Y_{s+\delta(s)}^{(j)},
Z_{t+\zeta(t)}^{(j)})$ satisfies the Lipschitz condition as well as
the integrable condition since $f_j$ satisfies (H1), (H2). Thus BSDE
(3.4) has a unique adapted solution.

Moreover, it is obvious that $(Y_t^{(j)}, Z_t^{(j)})_{t\in [t_i,
t_{i-1}]}$ satisfies both  ABSDE (3.3) and BSDE (3.4). Then from the
uniqueness of ABSDE's solution and that of BSDE's, we can easily
obtain the desired equalities.
\end{proof}

\begin{theorem}\label{Theorem 3.1}
Let $(Y^{(j)}, Z^{(j)})\in S_\mathcal{F}^2(0, T+K; \mathbb{R})\times
L_\mathcal{F}^2(0, T+K; \mathbb{R}^d)$ $(j=1, 2)$ be the unique
solutions to ABSDEs (3.1) respectively. If
\item{(i)} $\xi_s^{(1)}\geq \xi_s^{(2)}, s\in [T, T+K], a.e., a.s.;$
\item{(ii)} for all $t\in [0, T]$, $(y, z)\in
\mathbb{R}\times \mathbb{R}^d,$ $\theta^{(j)}\in S_\mathcal{F}^2(t,
T+K; \mathbb{R})$ $(j=1, 2)$ such that $\theta^{(1)} \geq
\theta^{(2)},$ $\{\theta_{r}^{(j)}\}_{r\in[t, T]}$ is a continuous
semimartingale and $(\theta_{r}^{(j)})_{r\in [T,
T+K]}=(\xi_r^{(j)})_{r\in [T, T_K]}$,
$$
f_1(t, y, z, \theta_{t+\delta(t)}^{(1)}, \eta_{t+\zeta(t)}^{(1)})
\geq f_2(t, y, z, \theta_{t+\delta(t)}^{(2)},
\eta_{t+\zeta(t)}^{(2)}),\ \ a.e., a.s.,\eqno(3.5)
$$
$$
f_1(t, y, z, \theta_{t+\delta(t)}^{(1)}, \frac{d\langle
\theta^{(1)}, B\rangle_r}{d r}|_{r=t+\zeta(t)}) \geq f_2(t, y, z,
\theta_{t+\delta(t)}^{(2)}, \frac{d\langle \theta^{(2)},
B\rangle_r}{d r}|_{r=t+\zeta(t)}),\ \ a.e., a.s.,\eqno(3.6)
$$
$$
f_1(t, y, z, \xi_{t+\delta(t)}^{(1)}, \frac{d\langle \theta^{(1)},
B\rangle_r}{d r}|_{r=t+\zeta(t)}) \geq f_2(t, y, z,
\xi_{t+\delta(t)}^{(2)}, \frac{d\langle \theta^{(2)}, B\rangle_r}{d
r}|_{r=t+\zeta(t)}),\ \ a.e., a.s.,\eqno(3.7)
$$
then $Y_t^{(1)} \geq Y_t^{(2)}, a.e., a.s..$

Moreover, the following holds:
\[
Y_0^{(1)}=Y_0^{(2)}\Leftrightarrow \left\{
\begin{tabular}{ll}
$\xi_T^{(1)}=\xi_T^{(2)};$ & $$\\
$f_1(t, Y_t^{(2)}, Z_t^{(2)}, Y_{t+\delta(t)}^{(1)},
Z_{t+\zeta(t)}^{(1)})=f_2(t, Y_t^{(2)}, Z_t^{(2)},
Y_{t+\delta(t)}^{(2)}, Z_{t+\zeta(t)}^{(2)}),\ \ t\in [0, T].$
\end{tabular}
\right.
\]
\end{theorem}

\begin{proof}
Consider the ABSDE (3.1) one time interval by one time interval. For
the first step, we consider the case when $t\in [t_1, T]$. According
to Proposition \ref{Proposition 3.2}, we can equivalently consider
the following BSDE:
\begin{center} $ \tilde{Y}_t^{(j)}=\xi_T^{(j)}+\int_t^T f_j(s,
\tilde{Y}_s^{(j)}, \tilde{Z}_s^{(j)}, \xi_{s+\delta(s)}^{(j)},
\eta_{s+\zeta(s)}^{(j)})ds-\int_t^T \tilde{Z}_s^{(j)}dB_s, $
\end{center}
from which we have
$$
\tilde{Z}_t^{(j)}=\frac{d\langle \tilde{Y}^{(j)}, B\rangle_t}{d t},\
\ t\in [t_1, T].\eqno(3.8)
$$
Noticing that $\xi^{(j)}\in S_\mathcal{F}^2(T, T+K; \mathbb{R})$
$(j=1, 2)$ and $\xi^{(1)}\geq \xi^{(2)},$ from (3.5) in (ii), we can
get, for $s\in [t_1, T],$ $y\in \mathbb{R}$, $z\in \mathbb{R}^d,$
\begin{center}
$f_1(s, y, z, \xi_{s+\delta(s)}^{(1)}, \eta_{s+\zeta(s)}^{(1)}) \geq
f_2(s, y, z, \xi_{s+\delta(s)}^{(2)}, \eta_{s+\zeta(s)}^{(2)}). $
\end{center}
According to the comparison theorem for $1$-dimensional BSDEs, we
can get
\begin{center}
$ \tilde{Y}_t^{(1)}\geq \tilde{Y}_t^{(2)},\ \ t\in[t_1, T],\ \
a.e.,a.s. $
\end{center}
 as well as
\[
Y_{t_1}^{(1)}=Y_{t_1}^{(2)}\Leftrightarrow \left\{
\begin{tabular}{ll}
$\xi_T^{(1)}=\xi_T^{(2)};$ & $$\\
$f_1(t, \tilde{Y}_t^{(2)}, \tilde{Z}_t^{(2)},
\xi_{t+\delta(t)}^{(1)}, \eta_{t+\zeta(t)}^{(1)})=f_2(t,
\tilde{Y}_t^{(2)}, \tilde{Z}_t^{(2)}, \xi_{t+\delta(t)}^{(2)},
\eta_{t+\zeta(t)}^{(2)}),\ \ t\in [t_1, T].$
\end{tabular}
\right.
\]
Consequently,
$$
Y_t^{(1)}\geq Y_t^{(2)},\ \ t\in[t_1, T+K],\ \ a.e.,a.s..\eqno(3.9)
$$

For the second step, we consider the case when $t\in [t_2, t_1]$.
Similarly, according to Proposition \ref{Proposition 3.2}, we can
consider the following BSDE equivalently:
\begin{center}
$ \tilde{Y}_t^{(j)}=Y_{t_1}^{(j)}+\int_t^{t_1} f_j(s,
\tilde{Y}_s^{(j)}, \tilde{Z}_s^{(j)}, Y_{s+\delta(s)}^{(j)},
Z_{s+\zeta(s)}^{(j)})ds-\int_t^{t_1} \tilde{Z}_s^{(j)}dB_s, $
\end{center}
from which we have $\tilde{Z}_t^{(j)}=\frac{d\langle
\tilde{Y}^{(j)}, B\rangle_t}{d t}$ for $t\in [t_2, t_1]$. Noticing
(3.8) and (3.9), according to (ii), we have, for $s\in [t_2, t_1]$,
$y\in \mathbb{R}$, $z\in \mathbb{R}^d,$
\begin{center}
$f_1(s, y, z, Y_{s+\delta(s)}^{(1)}, Z_{s+\zeta(s)}^{(1)}) \geq
f_2(s, y, z, Y_{s+\delta(s)}^{(2)}, Z_{s+\zeta(s)}^{(2)}). $
\end{center}
Applying the comparison theorem for BSDEs again, we can finally get
\begin{center}
$ Y_t^{(1)}\geq Y_t^{(2)},\ \ t\in[t_2, t_1],\ \ a.e.,a.s. $
\end{center} as well as
\[
Y_{t_2}^{(1)}=Y_{t_2}^{(2)}\Leftrightarrow \left\{
\begin{tabular}{ll}
$Y_{t_1}^{(1)}=Y_{t_1}^{(2)};$ & $$\\
$f_1(t, \tilde{Y}_t^{(2)}, \tilde{Z}_t^{(2)}, Y_{t+\delta(t)}^{(1)},
Z_{t+\zeta(t)}^{(1)})=f_2(t, \tilde{Y}_t^{(2)}, \tilde{Z}_t^{(2)},
Y_{t+\delta(t)}^{(2)}, Z_{t+\zeta(t)}^{(2)}),\ \ t\in [t_2, t_1].$
\end{tabular}
\right.
\]

Similarly to the above steps, we can give the proofs for the other
cases when $t\in [t_3, t_2],$ $[t_4, t_3],$ $\cdots,$ $[t_N,
t_{N-1}].$
\end{proof}

\begin{example}\label{Example 3.1}
Now suppose that we are facing with the following two ABSDEs:
$$\left\{
\begin{tabular}{rll}
$-dY_t^{(1)}=$ & $E^{\mathcal{F}_t}[Y_{t+\delta(t)}^{(1)}+\sin
(2Y_{t+\delta(t)}^{(1)})+|Z_{t+\zeta(t)}^{(1)}|+2]dt-Z_t^{(1)}dB_t,
$ & $
t\in[0, T];$\\
$Y_t^{(1)}=$ & $\xi_t^{(1)}, $ & $t\in[T, T+K];$\\
$Z_t^{(1)}=$ & $\eta_t^{(1)}, $ & $t\in[T, T+K],$
\end{tabular}\right.$$
$$
\left\{
\begin{tabular}{rll}
$-dY_t^{(2)}=$ & $E^{\mathcal{F}_t}[Y_{t+\delta(t)}^{(2)}+2 |\cos
Y_{t+\delta(t)}^{(2)}|+\sin Z_{t+\zeta(t)}^{(2)}-2]dt-Z_t^{(2)}dB_t,
$ & $
t\in[0, T];$\\
$Y_t^{(2)}=$ & $\xi_t^{(2)}, $ & $t\in[T, T+K];$\\
$Z_t^{(2)}=$ & $\eta_t^{(2)}, $ & $t\in[T, T+K],$
\end{tabular}\right.
$$
where $\xi_t^{(1)}\geq \xi_t^{(2)}, t\in [T, T+K].$

As both the generators depend on the anticipated term of $Z$ and
neither of them is increasing in the anticipated term of $Y$, we
cannot apply Peng, Yang's comparison theorem to compare $Y^{(1)}$
and $Y^{(2)}.$ While the following holds true:
\begin{center}
$x+\sin (2x)+|u|+2 \geq y+2 |\cos y| + \sin v-2,\ for\ all\ x\geq y,
x, y\in \mathbb{R}, u, v \in \mathbb{R}^d,$
\end{center}
which implies (3.5)-(3.7), then according to Theorem \ref{Theorem
3.1}, we get $Y_t^{(1)}\geq Y_t^{(2)},\ a.e.,\ a.s..$
\end{example}

\begin{remark}\label{Remark 3.1}
By the same way, for the case when
$\delta=\zeta$, (3.5)-(3.7) can be replaced by (3.6) together with
\begin{center}
$ f_1(t, y, z, \xi_{t+\delta(t)}^{(1)}, \eta_{t+\zeta(t)}^{(1)})
\geq f_2(t, y, z, \xi_{t+\delta(t)}^{(2)},
\eta_{t+\zeta(t)}^{(2)}),\ \ a.e., a.s.. $
\end{center}
\end{remark}

\begin{remark}\label{Remark 3.2}
If $f_1$ and $f_2$ are independent of the
anticipated term of $Z$, then (3.5)-(3.7) reduces to
$$
f_1(t, y, z, \theta_{t+\delta(t)}^{(1)}) \geq f_2(t, y, z,
\theta_{t+\delta(t)}^{(2)}).\eqno(3.10)
$$
Note that this conclusion is just with respect to the ABSDEs (2.1).
\end{remark}

\begin{remark}\label{Remark 3.3}
The generators $f_1$ and $f_2$ will satisfy (3.10), if for all $t\in
[0, T],$ $y\in \mathbb{R},$ $z\in \mathbb{R}^d,$ $\theta\in
L_\mathcal{F}^2(t, T+K; \mathbb{R}),$ $r\in [t, T+K],$ $f_1(t, y, z,
\theta_r)\geq f_2(t, y, z, \theta_r),$ together with one of the
following:
\item{(i)} for
all $t\in [0, T],$ $y\in \mathbb{R},$ $z\in \mathbb{R}^d,$ $f_1(t,
y, z, \cdot)$ is increasing, i.e., $f_1(t, y, z, \theta_r)\geq
f_1(t, y, z, \theta_r^\prime),$ if $\theta \geq \theta^\prime,$
$\theta, \theta^\prime\in L_\mathcal{F}^2(t, T+K; \mathbb{R}),$
$r\in [t, T+K];$
\item{(ii)} for all $t\in [0, T],$ $y\in \mathbb{R},$ $z\in \mathbb{R}^d,$
$f_2(t, y, z, \cdot)$ is increasing, i.e., $f_2(t, y, z,
\theta_r)\geq f_2(t, y, z, \theta_r^\prime),$ if $\theta \geq
\theta^\prime,$ $\theta, \theta^\prime\in L_\mathcal{F}^2(t, T+K;
\mathbb{R}),$ $r\in [t, T+K].$

Note that the latter is just the case that Peng-Yang [6] discussed
(see Theorem \ref{Theorem 2.2}).
\end{remark}

\begin{remark}\label{Remark 3.4}
The generators $f_1$ and $f_2$ will satisfy
(3.10), if
\begin{center}
$ f_1(t, y, z, \theta_r)\geq \tilde{f}(t, y, z, \theta_r)\geq f_2(t,
y, z, \theta_r), $
\end{center}
for all $t\in [0, T],$ $y\in \mathbb{R},$ $z\in \mathbb{R}^d,$
$\theta \in L_\mathcal{F}^2(t, T+K; \mathbb{R}), r\in [t, T+K].$
Here the function $\tilde{f}(t, y, z, \cdot)$ is increasing, for all
$t\in [0, T],$ $y\in \mathbb{R},$ $z\in \mathbb{R}^d,$
  i.e., $\tilde{f}(t, y, z, \theta_r)\geq \tilde{f}(t, y, z,
\theta_r^\prime),$ if $\theta_r\geq \theta_r^\prime,$ $\theta,
\theta^\prime\in L_\mathcal{F}^2(t, T+K; \mathbb{R}), r\in [t,
T+K].$
\end{remark}

\begin{example}\label{Example 3.2} The following three functions satisfy
the conditions in Remark \ref{Remark 3.4}: $f_1(t, y, z,
\theta_r)=E^{\mathcal{F}_t}[\theta_r+2 \cos \theta_r+1]$,
$\tilde{f}(t, y, z, \theta_r)=E^{\mathcal{F}_t}[\theta_r+\cos
\theta_r]$, $f_2(t, y, z, \theta_r)=E^{\mathcal{F}_t}[\theta_r+\sin
(2\theta_r)-2].$
\end{example}

\section*{Acknowledgements}

The author is grateful to professor Shige Peng for his encouragement
and helpful discussions.



\end{document}